\documentclass[12pt]{amsart}

\usepackage{amssymb,euscript,amsmath, mathrsfs,amscd}
\usepackage[dvips]{graphicx}
\usepackage[dvips]{color}
\usepackage{enumerate}
\usepackage{ytableau}
\newcounter{ENUM}
\newcounter{iENUM}
\newcommand{\bigcupdot}{\ensuremath{\mathaccent\cdot\cup}}
\newcommand{\beas}{\begin{eqnarray*}}
\newcommand{\eeas}{\end{eqnarray*}}
\newcommand{\bea}{\begin{eqnarray}}
\newcommand{\eea}{\end{eqnarray}}
\newcommand{\beq}{\begin{equation}}
\newcommand{\eeq}{\end{equation}}
\newcommand{\itm}{\item}
\newcommand{\be}{\begin{enumerate}}
\newcommand{\ee}{\end{enumerate}}
\newcommand{\fall}{\mathrm{for\ all}\ }
\newcommand{\nmo}{\frac 12(n-1)}

\newenvironment{ilist}{\renewcommand{\theENUM}{\roman{iENUM}}\renewcommand{\itm}{\addtocounter{iENUM}{1}\item[(\theENUM)]}\begin{itemize}\setcounter{iENUM}{0}}{\end{itemize}}
\newenvironment{alist}[1][0]{\renewcommand{\theENUM}{\alph{ENUM}}\renewcommand{\itm}{\addtocounter{ENUM}{1}\item[\theENUM)]}\begin{itemize}\setcounter{ENUM}{#1}}{\end{itemize}}

\newtheorem{thm}{Theorem}[section]
\newtheorem{prop}[thm]{Proposition}
\newtheorem{lem}[thm]{Lemma}
\newtheorem{cor}[thm]{Corollary}

\theoremstyle{definition}
\newtheorem{defn}[thm]{Definition}

\newtheorem{ex}[thm]{Example}

\theoremstyle{remark}
\newtheorem{note}[thm]{Note}

\newtheorem{rem}[thm]{Remark}

\numberwithin{equation}{section}

\def\P{{\mathbb{P}}}

\newcommand{\bm}[1]{{\boldsymbol{#1}}}
\newcommand{\st}{\,:\,}
\def\0{\bm{0}}

\def\iff{\qquad \Longleftrightarrow \qquad}

\def\Split{\operatorname{Split}}
\def\rind{\operatorname{RI}}

\def\ones{\operatorname{\#Ones}}

\def\Add{\operatorname{Add}}

\def\L{\mathcal L}
\def\cR{\mathcal R}
\def\cT{\mathcal T}

\subjclass[2010]{05A15}

\begin{document}
\title[A Distributive Lattice]{A Distributive Lattice Connected with
  Arithmetic Progressions of Length Three}
\author{Fu Liu}
\thanks{Fu Liu is partially supported by a grant from the Simons Foundation \#245939 and by NSF grant DMS-1265702.} \address{Fu Liu, Department of Mathematics, University of
  California, Davis, One Shields Avenue, Davis, CA 95616 USA.}
\email{fuliu@math.ucdavis.edu}
\date{\today}

\author{Richard P. Stanley}
\thanks{Richard Stanley is partially supported by NSF grant DMS-1068625.}
\address{Richard Stanley, Department of Mathematics, M.I.T.,
  Cambridge, MA 02139 USA.}
\email{rstan@math.mit.edu}

\begin{abstract}
Let $\mathcal{T}$ be a collection of 3-element subsets $S$ of
$\{1,\dots,n\}$ with 
the property that if $i<j<k$ and $a<b<c$ are two 3-element subsets in
$S$, then there exists an integer sequence $x_1<x_2<\cdots<x_n$ such
that $x_i,x_j,x_k$ and $x_a,x_b,x_c$ are arithmetic progressions.
We determine the number of such collections $\mathcal{T}$ and the
number of them of maximum size. These results confirm two conjectures
of Noam Elkies.
\end{abstract}

\maketitle

\section{Introduction} \label{sec1}

This paper has its origins in a problem contributed by Ron Graham to
the Numberplay subblog of the \emph{New York Times} Wordplay blog
\cite{blog}. Graham asked whether it is always possible to two-color a
set of eight integers such that there is no monochromatic three-term
arithmetic progression. A proof was found by Noam Elkies. Let
$\binom{[n]}{3}$ denote the set of all three-element subsets of $[n]=
\{1,2,\dots,n\}$. Define two such subsets, say $i<j<k$ and $a<b<c$, to
be \emph{consistent} if there exist integers
$x_1<x_2<\cdots<x_n$ for which both $x_i,x_j,x_k$ and $x_a,x_b,x_c$
are arithmetic progressions. For instance, $1<2<3$ and $1<2<4$ are
obviously not consistent. 

Let us call a collection $S$ of three-element subsets of integers
\emph{valid} if any two elements of $S$ are consistent. For
instance, the valid subsets of $\binom{[4]}{3}$ are 
 \beq \emptyset\ \{123\}\ \{124\}\ \{134\}\ 
   \{234\}\ \{123,134\}\ \{123, 234\}\ \{124,234\}, \label{eq:validex} 
  \eeq
so eight in all. 

Elkies needed to generate all valid subsets of
$\binom{[8]}{3}$. Define $f(n)$ to be the number of valid subsets of
$\binom{[n]}{3}$, so $f(4)=8$ as noted above. Elkies needed to work
with the case $n=8$, but he first computed that for $n\leq 7$ there
are exactly $2^{\binom{n-1}{2}}$ such subsets, leading to the obvious
conjecture that this formula holds for all $n\geq 1$. Elkies then
verified this formula for $n=8$, using the list of valid subsets to
solve Graham's problem. He then checked that $f(n)=2^{\binom{n-1}{2}}$
for $n=9$ and $n=10$. In Theorem~\ref{thm1} we show that indeed
$f(n)=2^{\binom{n-1}{2}}$ for all $n\geq 1$.

Let $\sigma(n)$ be the size (number of elements) of the largest valid
subset of $\binom{[n]}{3}$. Elkies showed that 
\begin{equation}\label{equ:sigma} \sigma(n) =
  \left\{ \begin{array}{rl} m(m-1), &  
       n=2m\\ [.5em]
     m^2, & n=2m+1. \end{array} \right. 
 \end{equation}
Let $g(n)$ be the number of valid subsets of $\binom{[n]}{3}$ of
maximal size $\sigma(n)$. Equation~\eqref{eq:validex} shows that
$\sigma(4)=2$ and $g(4)=3$. Elkies also conjectured (stated slightly
differently) that 
   \begin{equation}\label{equ:g}
     g(n) =  \left\{ \begin{array}{rl} 2^{(m-1)(m-2)}(2^m-1), & 
       n=2m\\[.5em] 
     2^{m(m-1)}, & n=2m+1. \end{array} \right. 
 \end{equation}
We prove this conjecture in Section~\ref{sec5}.

The basic idea behind our two proofs is the following. After the
Numberplay posting appeared, some further discussion continued on the
domino email forum \cite{propp}. In particular, David desJardins
observed that 
distinct triples $i<j<k$ and $i'<j'<k'$ are inconsistent if and only
if either
   $$ i\leq i',\ j\geq j',\ k\leq k' $$
or  
   $$ i\geq i',\ j\leq j',\ k\geq k'. $$
(The proof is straightforward though somewhat tedious.)
Jim Propp then defined a
partial ordering $P_n$ on certain elements of $[n]\times [n]\times
[n]$ such that the valid subsets of $\binom{[n]}{3}$ are just the
antichains of $P_n$. Since the antichains of a poset $P$ are just the
maximal elements of order ideals of $P$, we get that $f(n)=\#L_n$,
where $L_n := J(P_n)$ denotes the (distributive) lattice of order ideals of
$P_n$.  
  One can also define a coordinate-wise partial ordering $M_n$ on the
  set of semistandard Young tableaux (SSYT) of shape
  $\delta_{n-1}:=(n-2,n-3,\dots,1)$ and largest part at most $n-1$.  
%Let $M_n$ be the coordinate-wise partial ordering $M_n$ on the set of semistandard Young tableaux (SSYT) of shape $\delta_{n-1}=(n-2,n-3,\dots,1)$ and largest part at most $n-1$. 
  We show that $L_n\cong M_n$ by observing that both are distributive
  lattices and then showing that their posets of join-irreducibles are
  isomorphic. See \cite[Thm.~3.4.1]{ec1} for the relevant result on
  distributive lattices. It is an immediate consequence of standard
  results about SSYT that $\#M_n=2^{\binom{n-1}{2}}$, so the
  conjecture on $f(n)$ follows. The proof for $g(n)$ is more
  complicated. Let $K_n$ be the subset of $M_n$ corresponding to
  maximum size antichains of $P_n$ with respect to the isomorphism
  $L_n\to M_n$. By a result of Dilworth,  $K_n$ is a sublattice of
  $M_n$ and is therefore distributive.  We then determine the
  join-irreducibles of $K_n$. They are closely related to the
  join-irreducibles of $M_{1+\lfloor n/2\rfloor}$, from which we are
  able to compute $g(n)=\#K_n$.

\section{The number of valid subsets}

We assume the reader is familiar with basic definitions and results
on posets and tableaux presented in \cite[Chapter 3]{ec1} and
\cite[Chapter 7]{ec2}. 
%We review some important concepts for your paper. 
Recall that for any graded poset $P$, its \emph{rank-generating
  function} is  
 \[ F(P,q) = \sum_{x \in P} q^{\mathrm{rank}(x)}.\]

In this paper, we write a tableau $T$ using ``English notation,'' so
the longest row is at the top. Write $T_{a,b}=c$ to mean that the
$(a,b)$-entry of $T$ is equal to $c$.

On April 17, 2013, Jim Propp posted on the Domino Forum \cite{propp}
the following statement.

%\marginpar{shall we write Jim Propp's statement more formally into a
%proof?} 
\begin{quote} 
 I don't know if this reformulation is helpful, but pairwise
  consistent sets are in bijection with antichains in the subposet of
  $[n]\times [n]\times [n]$ containing all the $(i,j,k)$'s that
  satisfy $i+j < n+1 < j+k$.
\end{quote}

\begin{quote} ($[n]\times [n] \times [n]$ is the set $\{(i,j,k)\st 1 \leq
  i,j,k \leq n\}$, ordered so that $(i,j,k)\leq (i',j',k')$ iff $i
  \leq i'$ and $j \leq j'$ and $k \leq k'$.)
\end{quote}

\begin{quote}
To see the bijection, just map $(i,j,k)$ to $(i,n+1-j,k)$.
\end{quote}

Propp's statement follows easily from the observation of David
desJardins mentioned in the previous section.

Denote Propp's poset by $P_n$. The order ideals of $P_n$ form a
distributive lattice $L_n=J(P_n)$ under inclusion
\cite[{\S}3.4]{ec1}. There is a simple bijection \cite[end of
  {\S}3.1]{ec1} between the order ideals and antichains of a finite
poset. Further, under this bijection, the size of an antichain of a poset $P$ is exactly the number of elements covered by the corresponding order ideal in $J(P).$ Hence,
\begin{align}
  f(n)=& \#L_n, \text{the number of elements of $L_n$;} \label{fn} \\
  \sigma(n)=& \max(\text{number of elements covered by $x$}: x \in L_n); \label{sigman} \\
  g(n)=& \text{the number of $x \in L_n$ such that $x$ covers $\sigma(n)$ elements.} \label{gn}
\end{align}

Recall from Section~\ref{sec1} that $M_n$ is the poset of all SSYT
(semistandard Young tableaux) of shape
$\delta_{n-1}=(n-2,n-3,\dots,1)$ and largest part at most $n-1$,
ordered componentwise.  For
$n \ge 2$ the poset $M_n$ is a distributive lattice, where join is
entrywise maximum and meet is entrywise minimum, since either of the
operations of maximum and minimum on the integers distributes over the
other. We remark that $M_1$ is an empty set, which is neither
interesting nor a distributive lattice. Hence throughout this paper,
we assume $n \ge 2$. Let $Q_n$ denote the poset of join-irreducibles
of $M_n$, so $M_n=J(Q_n)$.

\begin{thm}\label{thm:LnMn} For $n \ge 2,$
\[ L_n \cong M_n.\]
\end{thm}
We will show that $P_n\cong Q_n$ in Proposition \ref{prop:pnqn}.  Hence by the fundamental theorem
for finite distributive lattices \cite[Thm.~3.4]{ec1}, $L_n\cong
M_n$. Theorem \ref{thm:LnMn} follows. See Figure~\ref{fig1} for the
lattices $L_4$ and $M_4$. We have labelled the join-irreducibles of
$L_4$ by the corresponding elements of $P_4$. One can also confirm $f(4)=8,$ $\sigma(4)=2$ and $g(4)=3$ by applying \eqref{fn},  \eqref{sigman}, and \eqref{gn} to the figure.

\begin{figure}
\centering
 \centerline{\includegraphics[width=12cm]{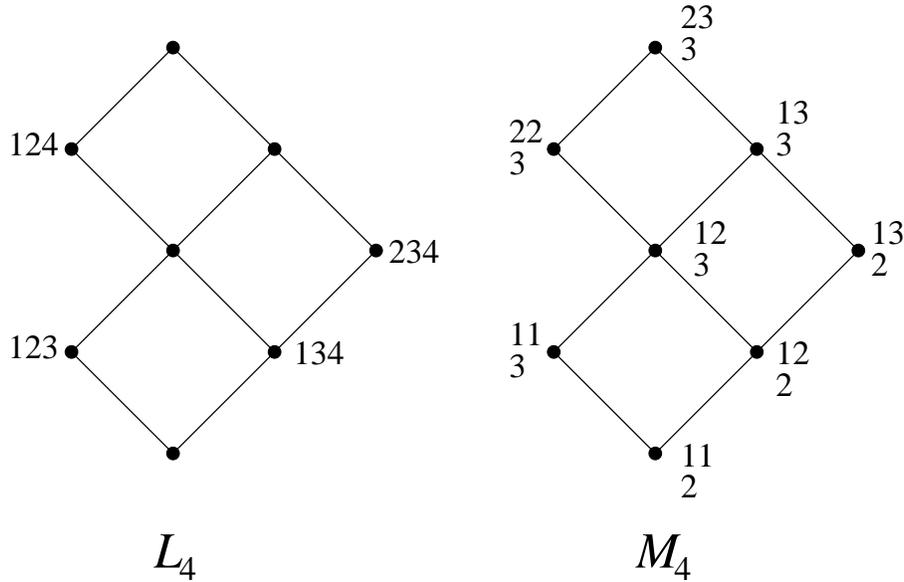}}
\caption{The distributive lattices $L_4$ and $M_4$}
\label{fig1}
\end{figure}

One main application of Theorem \ref{thm:LnMn} is that we are able to describe $f(n), \sigma(n)$ and
$g(n)$ as statistics related to the distributive lattice $M_n$. We say
that an entry $c$ of a tableau $T \in M_n$ is \emph{reducible} if by
replacing $c$ with $c-1$ in $T$ we obtain another tableau in $M_n$. 

\begin{note}\label{note:red}
The $(a,b)$-entry of a SSYT $T$ is reducible if and only if
 \[ T_{a,b}- T_{a, b-1} \ge 1 \qquad \text{ and } \qquad
 T_{a,b}-T_{a-1,b} \ge 2,\] 
where by convention we let for all $a,b$,
 \[ T_{a, 0} := a \quad \text{ \qquad and \qquad } T_{0,b}:= 0 \]  
(although they are not real entries in $T$).
\end{note}

\begin{cor}\label{cor:useMn} For $n \ge 2,$
 \be\item[\emph{(a)}]  
	$f(n) = \# M_n = \# J(Q_n)$.
     \item[\emph{(b)}]
	$\sigma(n)$ is the maximum number of reducible entries in $T$,
       for $T\in M_n$.
     \item[\emph{(c)}]
	$g(n)$ is the number of tableaux in $M_n$ that have the
       maximum number $\sigma(n)$ of reducible entries.
    \ee
\end{cor}

\begin{proof}
  It follows directly from Theorem \ref{thm:LnMn}, equations
  \eqref{fn}, \eqref{sigman} and \eqref{gn}, and the observation that
  for any $T \in M_n$, the number of elements covered by $T$ is the
  same as the number of reducible entries of $T.$ 
\end{proof}

In order to prove $P_n \cong Q_n$, we first need to describe the poset
$Q_n$ of join-irreducibles of $M_n$.

\begin{defn}\label{defn:add}
For any tableau $T$ with integer entries, we define $\Add(T; a, b, k)$
to be the tableau obtained from $T$ by adding $k$ to each
$(a',b')$-entry of $T$ with $(a',b') \ge (a,b)$. 

Let $\cT_{n-1}^0$ be the minimal element of $M_n$, so $\cT_{n-1}^0$
is the tableau of shape $\delta_{n-1}$ whose $(a,b)$-entry is
$a$. \end{defn} 

\begin{prop}\label{prop:Phin}
  Let $\Phi_n$ be the set $\{ (a, b, k) \in \P^3 \ | \ 1 \le k \le b
  \le n-1-a\}$ with the partial ordering  
  \[ (a,b,k) \le_{\Phi_n} (a', b', k') \text{ if } a \ge a',\  b \ge
  b',\ k \le k'.\] 
  
  Then for any $(a,b,k) \in \Phi_n,$ the tableau $\Add(\cT_{n-1}^0; a, b, k)$ is a
  join-irreducible of $M_n$. Moreover, all join-irreducible elements
  of $M_n$ are obtained in this way. 

  Furthermore, the map \[ \psi: (a,b,k) \mapsto \Add(\cT_{n-1}^0; a, b, k)\]
  induces a poset isomorphism from $\Phi_n$ to $Q_n$. 
  %\[ Q_n \cong \{ (a, b, k) \in \P^3 \ | \ 1 \le k \le b \le n-1-a\}.\]
\end{prop}
%\begin{lem}
%  For any positive integers $a, b, k$ satisfying $1 \le k \le b \le
%  n-1-a$, the tableau $\Add(\cT_{n-1}^0; a, b, k)$ is a
%  join-irreducible of $M_n$. Moreover, all join-irreducible elements
%  of $M_n$ are obtained in this way. 
%\end{lem}

\begin{ex} Let $n=4$. Then $\ytableausetup{centertableaux, boxsize=1em}
  \cT_3^0 = \begin{ytableau}
1 & 1 \\
2  
\end{ytableau}$, and the four join-irreducibles of $M_4$ are:
\begin{align*}
  \ytableausetup{centertableaux, boxsize=1.2em}
  \Add(\cT_3^0; 1, 1,1) = \begin{ytableau}
2 & 2 \\
3  
\end{ytableau}   \qquad & 
  \Add(\cT_3^0; 2, 1,1)= \begin{ytableau}
1 & 1 \\
3  
\end{ytableau} \\
  \Add(\cT_3^0; 1, 2,1) = \begin{ytableau}
1 & 2 \\
2 
\end{ytableau}   \qquad &
  \Add(\cT_3^0; 1, 2,2)= \begin{ytableau}
1 & 3 \\
2  
\end{ytableau} 
\end{align*}

\end{ex}

We need several preliminary results before proving Proposition \ref{prop:Phin}.
\begin{lem}\label{lem:compdiff}
  Suppose $T$ is an SSYT and $T_{a,b}$ is an entry of $T$. If 
  \begin{equation}\label{equ:compdiff1}
    T_{a,b} - T_{a-1,b} > T_{a,b-1}-T_{a-1,b-1},
  \end{equation}
  then $T_{a,b}$ is a reducible entry.  
\end{lem}

\begin{proof}
  We clearly have $T_{a,b}-T_{a-1,b} \ge 2$. Thus it is enough to
  show $T_{a,b} - T_{a,b-1} \ge 1,$ which follows from  
  $T_{a,b} - T_{a,b-1} > T_{a-1,b} - T_{a-1,b-1} \ge 0$.
\end{proof}

\begin{cor}\label{cor:compdiff1}
Suppose $T$ is an SSYT and $T_{a,b}$ is an entry of $T$. If 
  \begin{equation}\label{equ:compdiff2}
    T_{a,b} - T_{a-1,b} < T_{a,b-1}-T_{a-1,b-1},
  \end{equation}
  then there is a reducible entry in the $a$-th row of $T$.
\end{cor}

\begin{proof}
  \eqref{equ:compdiff2} implies that $T_{a,b-1} - T_{a-1,b-1} \ge 2 >
  1 = T_{a,0} - T_{a-1,0}$. Therefore there exists $1 \le b' \le b-1$
  such that 
   \[T_{a,b'} - T_{a-1,b'} > T_{a,b'-1}-T_{a-1,b'-1}.\]
   Then the conclusion follows from Lemma \ref{lem:compdiff}.
\end{proof}

\begin{cor}\label{cor:compdiff2}
  Suppose $T$ is an SSYT of shape $\lambda$. Then the following are
  equivalent. 
  \be
	  \itm[(i)] $T_{a_0,b_0}$ is the only reducible entry of $T$.
	  \itm[(ii)] For any pair of indices $(a,b)$ such that $T_{a,b}$ is
          an entry of $T$, we have the following:
  \be
	  \itm[(a)] $T_{a,b} - T_{a-1,b} = 1$ for any $(a,b)$ satisfying $a
          \neq a_0$ or else satisfying $a = a_0$ and $b < b_0$; 
	  \itm[(b)] $T_{a,b} - T_{a-1,b} \ge 2$ for $(a,b) = (a_0, b_0)$;
	  \itm[(c)] $T_{a,b} - T_{a-1,b} = T_{a,b-1}-T_{a-1,b-1}$ for any
          $(a,b)$ satisfying $a = a_0$ and $b > b_0$. 
  \ee
  \itm[(iii)] $T = \Add(\cT^0; a_0, b_0, k)$ for some $k \ge 1,$ where
  $\cT^0$ is the minimal SSYT of shape $\lambda$. 
  \ee
\end{cor}

\begin{proof}
 It is straightforward to verify that (ii) and (iii) are equivalent
 and that (iii) implies (i). We will show (i) implies (ii). Assuming
 (i), by Lemma \ref{lem:compdiff} and Corollary \ref{cor:compdiff1},
 we have
	\[ T_{a,b} - T_{a-1,b} = T_{a,b-1}-T_{a-1,b-1},\ \ \fall
        (a,b) \text{ with } a \neq a_0.\]
Note that 
\begin{equation}\label{equ:=1}
	T_{a,0}-T_{a-1,0} = a-(a-1) = 1, \qquad \fall a.
\end{equation}
We have $T_{a,b}-T_{a-1,b} = 1$ when $a \neq a_0$. Since $T_{0,b}=0$
for any $b,$ it follows that $T_{a,b} = a$ for any $a < a_0$. In
particular, 
\[ T_{a_0-1, b} = a_0-1, \qquad \fall b.\]
Since the entries in the $a$th row are weakly increasing, we must have
\begin{equation}\label{equ:compdiff3} T_{a_0,b} - T_{a_0-1,b} \ge
  T_{a_0,b-1}-T_{a_0-1,b-1}, \qquad \fall b.\end{equation}  
Thus it follows from Lemma \ref{lem:compdiff} that the equality of
\eqref{equ:compdiff3} holds when $b \neq b_0$ and inequality holds
when $b = b_0$. 
Finally, it follows from \eqref{equ:=1} that condition (ii)(a) holds
for $a=a_0$ and $b< b_0$ and thus condition (b) holds, completing the
proof.  
\end{proof}

\begin{proof}[Proof of Proposition \ref{prop:Phin}]
  We first verify the first part of the conclusion of the proposition,
  which is equivalent to say that $\psi$ gives a bijection from
  $\Phi_n$ to $Q_n.$ Because of Corollary 
  \ref{cor:compdiff2}, it is sufficient to show that given $a+b \le
  n-1$ (which implies that $(a,b)$ is an entry in a tableau of shape
  $\delta_{n-1})$, we have $k \le b$ if and only if all the entries in
  $\Add(\cT_{n-1}^0; a, b, k)$ are at most $n-1$. However, note that
  the entries in $\cT_{n-1}^0$ are less than $n-1$ and the largest
  entry in $\Add(\cT_{n-1}^0; a, b, k)$ that is different from
  $\cT_{n-1}^0$ is the last entry in the $b$th column of
  $\Add(\cT_{n-1}^0; a, b, k)$: 
  \[ \Add(\cT_{n-1}^0; a, b, k)_{n-1-b, b} = n-1-b + k.\]
  Therefore each entry in $\Add(\cT_{n-1}^0; a, b, k)$ is at most
  $n-1$ if and only if $n-1-b + k \le n-1,$ which is equivalent to $k
  \le b$. Thus the map $(a,b,k) \mapsto \Add(\cT_{n-1}^0; a, b, k)$
  induces a bijection from $\Phi_n$ to $Q_n$. 

  It is easy to see that $\Add(\cT_{n-1}^0; a, b, k) \le \Add(\cT_{n-1}^0; a', b', k')$ if and only if $a \ge a', b
  \ge b', k \le k'$. Hence we get an isomorphism, as desired.
\end{proof}

We have the following corollary to Proposition \ref{prop:Phin} which will be used later. 
\begin{cor}
For $n \ge 2,$
  \begin{equation}\label{equ:diffQn}
  \# Q_{n+1} - \# Q_{n} = \binom{n}{2}.
\end{equation}
\end{cor}
\begin{proof}
  \[ \# Q_{n} = \sum_{k=1}^{n-2} \sum_{b=k}^{n-2} (n-1-b) = \sum_{k=1}^{n-2} \binom{n-k}{2} = \sum_{\alpha=1}^{n-2} \binom{\alpha+1}{2}.\]
\end{proof}

Since we've shown that $Q_n \cong \Phi_n$ in Proposition \ref{prop:Phin}, we establish that $P_n\cong Q_n$ by showing $P_n \cong \Phi_n.$

\begin{prop} \label{prop:pnqn} 
  Let $\Phi_n$ be the poset defined in Proposition \ref{prop:Phin}. Define a map $\varphi\colon \Phi_n\to P_n$ by
  $$ \varphi(a, b, k)= (k,n-b,n+1-a). $$
Then $\varphi$ is an isomorphism of posets.

Therefore, $P_n \cong \Phi_n \cong Q_n.$
\end{prop} 

\begin{proof}
This is just a straightforward verification. First we check that
$\varphi(a,b,k)\in P_n$. We need to show that
  $$ 1\leq k \leq n,\ \ 1\leq n-b\leq n,\ \ 1\leq n+1-a\leq n $$ 
  $$ k+n-b<n+1<2n+1-a-b. $$
These inequalities are immediate from $(a, b, k) \in \Phi_n,$ i.e., $1 \le k \le b \le n-1-a$ and $a, b, k \in \P.$
We can then check that $\varphi^{-1}(i,j,\ell) =
(n+1-\ell,n-j,i)$. Hence $\varphi$ is a bijection $\Phi_n\to P_n$.

It remains to show that $\varphi$ is a poset isomorphism. However, one checks directly that $(a,b,k) \le_{\Phi_n} (a', b', k')$ if and only if $k \le k', n-b \le n-b', n+1-a \le n+1-a',$ i.e., $\varphi(a,b,k) \le_{P_n} \varphi(a',b',k').$
\end{proof}

Therefore, as we discussed before, Theorem \ref{thm:LnMn} follows from the above proposition. %Below we apply Corollary \ref{cor:useMn} (a) to find $f(n).$

\begin{thm} \label{thm1}
For any $n \ge 2,$ the rank-generating function of $M_n$ is given by
\bea\label{equ:genMn}
  F(M_n,q) & = & (1+q)^{n-2}(1+q^2)^{n-3}\cdots (1+q^{n-2})\\ & = &
  \prod_{i=1}^{n-2} (1+q^i)^{n-1-i},  
\eea
where $F(M_2,q) = 1$.
Hence we have $$f(n)=2^{\binom{n-1}{2}}. $$
\end{thm}

\begin{proof}
%We can count the number of elements of $M_n$ (in fact, its
%rank-generating function) using standard results from the theory of
%symmetric functions. 
We compute $F(M_n, q)$ using standard results from the theory of
symmetric functions. The rank of an element (SSYT) in $M_n$ is the sum 
of its entries minus $\binom{n}{3}$. Denote the rank-generating
function of $M_n$ by $F(M_n,q)$. 
If $T$ is an SSYT with $m_i$
entries equal to $i$, write $x^T=x_1^{m_1}x_2^{m_2}\cdots$. The Schur
function $s_{\delta_{n-2}}(x_1,\dots,x_{n-1})$ may be defined (see
\cite[Def.~7.10.1]{ec2}) as
  $$ s_{\delta_{n-2}}(x_1,\dots,x_{n-1})=\sum_T x^T, $$
where $T$ ranges over all SSYT of shape $\delta_{n-2}$ and largest
part at most $n-1$. Hence
  $$ q^{\binom n3}F(M_n,q) = s_{\delta_{n-2}}(q,\dots,q^{n-1}). $$
Now we have \cite[Exer.~7.30(a)]{ec2} 
  \beq s_{\delta_{n-2}}(x_1,\dots,x_{n-1})=\prod_{1\leq i<j\leq n-1}
       (x_i+x_j). \label{eq:sprod} \eeq 
From this equation it is immediate that
  $$ F(M_n,q) = (1+q)^{n-2}(1+q^2)^{n-3}\cdots (1+q^{n-2}). $$
Setting $q=1$ gives $f(n)=\#L_n=\#M_n= 2^{\binom{n-1}{2}}$, completing
the proof. 
\end{proof}

\textsc{Note.} Rather than using the special formula \eqref{eq:sprod}
we could have used the hook-content formula \cite[Thm.~7.21.2]{ec2}
for $s_{\lambda}(1,q,\dots,q^{m-1})$, valid for any $\lambda$ and $m$. (Note
that if $\lambda$ is a partition of $N$, then $q^N
s_{\lambda}(1,q,\dots,q^{m-1}) =s_{\lambda}(q,q^2,\dots,q^m)$.)
Equation~\eqref{eq:sprod} can be translated into some enumerative
property of valid subsets of $\binom{[n]}{3}$, but it seems rather
contrived.

\section{Valid subsets of maximum size}

In the rest of the paper, we will prove Elkies' conjecture on the
formula \eqref{equ:g} for $g(n)$ as well as provide another proof for
his formula \eqref{equ:sigma} for $\sigma(n)$. Recall that in Corollary
\ref{cor:useMn} we give alternative definitions for $\sigma(n)$ and
$g(n)$ in terms of $M_n$. We find it is convenient to %let  
%\[ T_{a, 0} := 0 \quad \fall a \text{ \qquad and \qquad } T_{0,b}:= 0 \quad \fall b\] 
%(although they are not really entries in $T$) and 
use the following obvious lemma to describe tableaux in $M_n$ using
inequalities.  
\begin{lem}\label{lem:Mnentry} 
Suppose $T$ is a tableau of shape $\delta_{n-1}$ with integer
entries. Then $T \in M_n$ if and only if the entries of $T$ satisfy
the following conditions:  
	\be
\itm[\emph{(a)}] $T_{1,1} \ge 1$. 
\itm[\emph{(b)}] $T_{a,b} - T_{a,b-1} \ge 0,$ for any $2 \le b \le n-2,\ 1\le
a \le n-1-b$ (weakly increasing on rows) 
\itm[\emph{(c)}] $T_{a,b} - T_{a-1,b} \ge 1,$ for any $2 \le a \le n-2,\ 1\le
b \le n-1-a$ (strictly increasing on columns) 
\itm[\emph{(d)}] $T_{n-1-b,b} \le n-1,$ for any $1 \le b \le n-2$. 
%for any $a,b$ satisfying $a+b=n-1$.
	\ee
\end{lem}

\begin{rem}\label{rem:ssyt}
  For convenience, we sometimes abbreviate conditions (a)~--~(c) of
  Lemma~\ref{lem:Mnentry} as: $\fall 1 \le b \le n-2, \ 1 \le a \le
  n-1-b,$
  \[ T_{a,b}- T_{a, b-1} \ge 0 \quad \text{ and } \quad T_{a,b}-T_{a-1,b} \ge 1, \qquad , \] 
with the convention $T_{a,0}=a$ and $T_{0,b}=0$.
\end{rem}

%Since the enumeration of $g(n)$ is related to $\sigma(n)$, our
%discussion also leads to a proof for Elkies' formulas for $\sigma(n)$
%using the distributive lattice $M_n$.   
Since $\sigma(n)$ is the maximum possible number of reducible entries
in a tableau in $M_n,$ we first give an upper bound for the number of
reducible entries in $T \in M_n$.

\begin{lem}\label{lem:upb}
Let $T \in M_n$. Then for any $1 \le b \le n-2,$ 
\[ \# \text{reducible entries in the $b$th column of $T$} \le \min(b,
n-1-b).\]% = \begin{cases}b & 1 \le b \le n/2 \\ n-1-b
         % &  \end{cases}.\] 
Therefore,
\[ \# \text{reducible entries in $T$} \le \sum_{b=1}^{n-2} \min(b, n-1-b).\]
\end{lem}

\begin{proof}
First, the number of reducible entries in the $b$th column is at most
the number of entries in the $b$th column, which is $n-1-b$.  

By Lemma \ref{lem:Mnentry}(d), 
the last entry $T_{n-1-b,b}$ in the $b$th column satisfies 
\begin{align}
  n-1 \ge& \quad T_{n-1-b,b} = \sum_{a=1}^{n-1-b} \left(
  T_{a,b}-T_{a-1,b} \right) \label{equ:ineq2}. \end{align}
Hence by Remark \ref{rem:ssyt} and Note \ref{note:red},
 \begin{align}
   n-1 \ge& \quad \# \left(\text{irreducible entries in the $b$th
     column of $T$}\right)+ \label{equ:ineq1}\\  
	& \quad 2 \times \# \left(\text{reducible entries in the $b$th
    column of $T$}\right) \nonumber \\  
=& \quad \# \left(\text{entries in the $b$th column of $T$}\right)+
\nonumber \\ 
 & \quad \# \left(\text{reducible entries in the $b$th column of
  $T$}\right) \nonumber\\ 
=& \quad n-1-b+ \# \left(\text{reducible entries in the $b$th column
  of $T$}\right). \nonumber  
\end{align}
Then we conclude that the number of reducible entries in the $b$th
column is at most $b$. 
\end{proof}

\begin{defn}
	Let $K_n$ be the coordinate-wise partial ordering on the set
        of all the tableaux in $M_n$ that have $\sum_{b=1}^{n-2}
        \min(b, n-1-b)$ reducible entries. (Thus $K_n$ is a
        subposet of $M_n$.) 
\end{defn}

\begin{rem}\label{rem:elkies}
	By Lemma \ref{lem:upb}, we see that $\sigma(n)$ is at most 
	\[\sum_{b=1}^{n-2} \min(b, n-1-b) =\begin{cases}
		1 + 2 + \cdots + (m-1) + (m-1) + \cdots +2 + 1\\ \ \ =
                m(m-1),\qquad \text{if $n=2m$}; \\ 
		1 + 2 + \cdots + (m-1) + m + (m-1) + \cdots +2 +
                1\\ \ \ = m^2,\qquad\quad\quad \text{if
                  $n=2m+1$}.  
	\end{cases}\]

One only needs show that $K_n$ is nonempty to confirm Elkies' formula
\eqref{equ:sigma} for $\sigma(n)$. Although one can easily directly
construct a tableau that is in $K_n,$ we choose to start by analyzing
the properties of tableaux in $K_n$ and give a proof for the
nonemptyness of $K_n$ indirectly in the next section. The benefit of
doing this is that the arguments are useful for figuring out the
cardinality of $K_n$, which gives the value of $g(n)$.
%Instead we will do this by showing the cardinality of $K_n$ is given
%by the numbers provided on the right hand side of \eqref{equ:g},
%which are nonzero. Clearly, this gives us proofs for formulas for
%$\sigma(n)$ and $g(n)$. 
\end{rem}

\subsection*{Properties of tableaux in $K_n$}

We have the following immediate consequence of Lemma \ref{lem:upb} and
its proof. 
\begin{lem}\label{lem:charKn}
Suppose $T \in M_n$. Then 
%the number reducible entries in $T$ is $\sum_{b=1}^{n-2} \min(b, n-1-b)$
$T \in K_n$ if and only if the following two conditions are satisfied.
 \be
	\itm[\emph{(a)}]\label{itm:left} For any $1 \le b \le
        \nmo,$ the 
        number of reducible entries in the $b$th column of $T$ is $b$.
        %The $b$th column of $T$ has $b$ reducible entries.% and
        %$n-1-2b$ irreducible entries. 
      	\itm[\emph{(b)}] \label{itm:right} For any $\nmo < b \le
          n-2,$ all the entries in the $b$th column of $T$ are
          reducible.  
 \ee
\end{lem}

While condition (b) of the above lemma is enough for us to determine
how to create the right half entries of tableaux in $K_n,$ we will
discuss explicit conditions for the left half entries as a consequence
of Lemma \ref{lem:charKn}(a) in several corollaries below.  

\begin{cor}\label{cor:charleftKn0}
Suppose $T \in K_n$. Then for any $1 \le b \le \nmo$,
\be
  \itm[\emph{(a)}] \label{itm:last} the last entry $T_{n-1-b, b}$ in the $b$th
  column of $T$ is $n-1;$  
  \itm[\emph{(b)}]\label{itm:diff} for any $1 \le a \le n-1-b,$ the entry
  $T_{a,b}$ is reducible if and only if $T_{a,b}-T_{a-1,b} = 2$, and
  $T_{a,b}$ is irreducible if and only if $T_{a,b}-T_{a-1,b} = 1$. 
\ee
Therefore, 
\be %[2]
  \itm[\emph{(c)}] among all the $n-1-b$ entries $T_{a,b}$ in the
  $b$th column of 
  $T,$ there are $b$ entries satisfying $T_{a,b}-T_{a-1,b}=2,$ and the
  remaining $n-1-2b$ entries satisfying $T_{a,b} - T_{a-1,b}=1$. 
\ee
\end{cor}

\begin{proof}
  By Lemma \ref{lem:charKn}(a) the number of reducible
  entries in the $b$th column of $T$ is $b$.  
  However, by the proof of Lemma \ref{lem:upb}, one sees that this
  only happens when the equalities in both \eqref{equ:ineq1} and
  \eqref{equ:ineq2} hold. Therefore (a) and (b) follow, and then (c)
  follows. 
\end{proof}

\begin{cor}\label{cor:charleftKn1}
  Suppose $T \in K_n$. Then for any $1 \le b \le \nmo-1$ and $1 \le
  a \le n-1-b$ (so both $T_{a,b}$ and $T_{a-1, b+1}$ are on the left
  half of $T$), we have 
\begin{equation}\label{equ:diag}
T_{a, b} - T_{a-1, b+1} \le 1.
\end{equation}
\end{cor}

\begin{proof}
  Assume to the contrary that 
  \begin{equation}\label{equ:diff1}
    T_{a, b} - T_{a-1, b+1} \ge 2.
  \end{equation}
  We claim that $a+b < n-1$. If $a=1,$ then $1+b \le \nmo < n-1;$
  if $a > 1$, then by Corollary \ref{cor:charleftKn0}(a) the
  condition $a 
  + b = n-1$ implies $T_{a,b}= n-1=T_{a-1,b+1}$, which is
  impossible. Thus $a+b < n-1,$ and so $a+(b+1) \le n-1$. Hence  $T$
  has an $(a, b+1)$-entry. Then by Corollary \ref{cor:charleftKn0}(b),  
\begin{equation}\label{equ:diff2}
T_{a,b+1} - T_{a-1, b+1} \le 2.
\end{equation}
Comparing with our assumption \eqref{equ:diff1}, we conclude that
$T_{a,b+1} \le T_{a, b}$. However, since $T$ is a SSYT, one has to
have  
\begin{equation}\label{equ:diff3}
  T_{a,b+1} = T_{a, b}.
\end{equation}
Thus both equalities in \eqref{equ:diff1} and \eqref{equ:diff2}
hold. In particular, we get $T_{a,b+1} - T_{a-1, b+1} =2$. Now using
Corollary \ref{cor:charleftKn0}(b) we conclude that $T_{a, b+1}$ is
reducible. However, by Note \ref{note:red} this implies that
$T_{a,b+1} - T_{a,b} \ge 1,$ which contradicts 
equation~\eqref{equ:diff3}. 
\end{proof}

It turns out that \eqref{equ:diag} is an important property. Since it
is related to the difference of two consecutive northeast-southwest
diagonal entries, we often refer to it as ``the diagonal
property.'' Below we give an easy but useful lemma for using this
property.   
\begin{lem}\label{lem:diag}
  Suppose $T$ is a tableau (of some shape) filled with integer entries. %Assume $T$ has $(a,b)$, $(a+1,b)$ and $(a,b+1)$-entries and they satisfy
  Assume $b \ge 1$ and $a \ge 0$ and the $(a,b)$, $(a+1,b)$ and
  $(a,b+1)$-entries of $T$ satisfy 
  \begin{equation}\label{equ:two}
  T_{a+1, b} - T_{a, b+1} \le 1 \qquad \text{ and } \qquad T_{a+1, b} - T_{a,b} \ge 1.
\end{equation}
Then 
\[ T_{a,b} \le T_{a, b+1},\]
where the equality holds if and only if $T_{a+1, b} = T_{a, b} + 1 = T_{a, b+1}+1$.
\end{lem}

This lemma says that the diagonal property together with the property
of strictly increasing on columns implies the property of weakly
increasing on rows. 

\begin{proof}
  We combine the two inequalities in \eqref{equ:two}: 
  \[ T_{a,b}+1 \le T_{a+1, b} \le T_{a, b+1} +1.\]
  Then the conclusion follows.
\end{proof}

\begin{cor}\label{cor:charleftKn2}
  Suppose $T \in K_n$. Then   
\begin{equation}\label{equ:=a}
T_{a, b} = a, \qquad \fall 1 \le b \le \nmo, \ 0 \le a \le \nmo-b.
\end{equation}
Therefore for any $1 \le b \le \nmo$ and $1 \le a \le \nmo-b,$
the entry $T_{a,b}$ is irreducible. Hence the first $\lfloor \nmo
\rfloor - b$ entries (not counting $T_{0,b}$) in the $b$th column are
irreducible for any $1 \le b \le \lfloor \nmo \rfloor$. 
\end{cor}
\begin{proof}
  We prove \eqref{equ:=a} by induction on $a$ noting that the indicies
  in \eqref{equ:=a} can be described as  
\[0 \le a \le \nmo -1, \ 1 \le b \le \nmo-a.\]
  The base case when $a=0$ clearly holds since $T_{0,b}=0$ by our
  convention. Suppose \eqref{equ:=a} holds for $a = a_0$ for some $0
  \le a_0 \le \nmo-2$. We want to show $T_{a,b}=a,$ for $a = a_0 +
  1$ and any $1 \le b \le \nmo -a$. One checks that $a$ and $b$
  satisfy 
  \[ 1 \le b \le \nmo -1 \quad \text{ and } \quad 1 \le a \le \nmo-b \le n-1-b.\]
  Hence by Corollary \ref{cor:charleftKn1}, we have the diagonal property 
  \[ T_{a, b} - T_{a-1, b+1} \le 1.\]
Meanwhile, since $T$ is an SSYT, we have
\[ T_{a, b} - T_{a-1,b} \ge 1.\] 
However, by the induction hypothesis, 
\[ T_{a-1, b} = a_0 = T_{a-1, b+1}.\]
It follows from Lemma \ref{lem:diag} that \[T_{a,b} = T_{a-1,b}+1 = (a-1)+1 =a.\] Hence \eqref{equ:=a} holds.

%Suppose $(a,b)$ satisfies $1 \le b \le \nmo -1$ and $1 \le a \le \nmo-b$, which are equivalent to $1 \le a \le \nmo -1$ and $1 \le b \le \nmo-a$. Using \eqref{equ:=a}, one sees that $T_{a,b}-T_{a-1,b}=1$. Hence the second conclusion follows from Corollary \ref{cor:charleftKn0}(b).
The second conclusion easily follows from \eqref{equ:=a} and Corollary \ref{cor:charleftKn0}(b).
\end{proof}

We are now ready to state and prove the main result of this section.
\begin{prop}\label{prop:charKn}
  Suppose $T$ is a tableau of shape $\delta_{n-1}$ filled with integer
  entries. Then $T \in K_n$ if and only if the following conditions
  are satisfied. 
	\be
	\itm[\emph{(a)}] For any $1 \le b \le \lfloor \nmo \rfloor,$ 
	\be
	  \itm[\emph{(i)}] for any $1 \le a \le \lfloor \nmo \rfloor-b,$
          we have $T_{a,b}=a;$ 
	\itm[\emph{(ii)}] among the $\lfloor n/2 \rfloor$ remaining
        values of $a$, viz., $\lfloor \nmo \rfloor-b+1 \le a \le
        n-1-b$, we have that $b$ of them 
        satisfy  $T_{a,b}-T_{a-1,b} = 2$, and the remaining $\lfloor
        n/2 \rfloor-b$ of them satisfy $T_{a,b}-T_{a-1,b} = 1$. 
      \ee
	  \itm[\emph{(b)}] For any $1 \le b \le \lfloor \nmo
          \rfloor-1$ and 
          $\lfloor \nmo\rfloor -b+1 \le a \le n-1-b,$ we have the
          diagonal property $T_{a,b} - T_{a-1,b+1} \le 1$. 
	  \itm[\emph{(c)}] For any $\lfloor \nmo \rfloor +1 \le b \le n-2$ and
          any $1 \le a \le n-1-b$, we have $T_{a, b} \le n-1,$
          $T_{a,b} - T_{a-1,b} \ge 2$ and $T_{a,b} - T_{a, b-1} \ge
          1$. 
	\ee
\end{prop}

\begin{proof}
  Suppose $T \in K_n$. It follows from Lemma \ref{lem:charKn},
  Corollaries \ref{cor:charleftKn0} and \ref{cor:charleftKn1}, and
  Note \ref{note:red} that (a)--(c) hold.  

  Now suppose (a)--(c) hold. We first show that $T \in M_n$ by verifying
  that the conditions in Lemma \ref{lem:Mnentry} are satisfied. It is
  clear that conditions (a) and (c) of Lemma \ref{lem:Mnentry}
  hold. One sees that (a) implies that $T_{n-1-b,b} = n-1$ for any $1
  \le b \le \lfloor \nmo \rfloor,$ which together with (c), implies
  condition (d) of Lemma \ref{lem:Mnentry}.
  
  Thus we only need to show that for any $2 \le b \le n-2,\ 1\le
a \le n-1-b$,
  \begin{equation}\label{equ:weak}
    T_{a,b}-T_{a,b-1} \ge 0.
  \end{equation}
However, we already know that \eqref{equ:weak} holds for $\lfloor \nmo
\rfloor +1 \le b \le n-2$ by condition (c), and holds for $2 \le b \le
\lfloor \nmo \rfloor$ and $1 \le a \le \lfloor \nmo\rfloor -b$
by condition (a)(i). Thus we assume $2 \le b \le \lfloor \nmo
\rfloor$ and $\lfloor \nmo\rfloor -b+1 \le a \le n-1-b$. 
  %Then 
  %\[ 1 \le b-1 \le \lfloor \nmo \rfloor-1, \quad  \lfloor
  %\nmo\rfloor -(b-1) \le a \le n-2-(b-1).\]  
  Hence by (b), we have the diagonal property
  \[ T_{a+1, b-1} - T_{a, b} \le 1.\]
Then \eqref{equ:weak} follows from Lemma \ref{lem:diag}, so we
conclude that $T \in M_n$.  
  
Now it suffices to verify conditions (a) and (b) of Lemma
\ref{lem:charKn} to conclude $T \in K_n$. Lemma \ref{lem:charKn}(b)
clearly follows from (c). For condition (a) of Lemma \ref{lem:charKn},
because of Lemma \ref{lem:upb}, it is enough to check that for any $1
\le b \le \lfloor \nmo \rfloor$ and $\lfloor \nmo \rfloor -b +1
\le a \le n-1-b,$ if $T_{a, b}- T_{a-1,b}=2,$ then $T_{a,b}$ is
reducible. However, by Note \ref{note:red}, it is sufficient to prove  
  \begin{equation}\label{equ:red}
    T_{a,b} - T_{a-1,b} =2 \qquad \Longrightarrow \qquad T_{a,b} - T_{a, b-1} \ge 1.
  \end{equation}
  If $b = 1,$ since $T_{a,b-1} = a = T_{a-1,b-1}+1 \le T_{a-1,b}+1$, we immediately
have $T_{a,b} - T_{a,b-1} \ge 1$. If $b > 1$ and $a =\lfloor \nmo
\rfloor -b +1,$ it follows from condition (a)(i) that
$T_{a,b-1}-T_{a-1,b} = a - (a-1) = 1$. Thus \eqref{equ:red} holds.  
  Hence we assume $2 \le b \le \lfloor\nmo\rfloor$ and $\lfloor
  \nmo \rfloor -b +2 \le a \le n-1-b$. %Then  
  %\[ 1 \le b-1 \le \lfloor \nmo \rfloor-1, \quad \lfloor \nmo \rfloor -(b-1) \le a-1 \le n-2-b < (n-2)-(b-1).\] 
Applying  (b) again, we get $T_{a, b-1} - T_{a-1, b} \le 1$. Then
\eqref{equ:red} follows. 
  %\[ T_{a, b} = T_{a-1, b} + 2 = T_{a-1,b} + 1 + 1 \ge T_{a,b-1} + 1.\]
\end{proof}

\section{Two sides}

Since the characterization of the left half and right half of tableaux
in $K_n$ stated in Proposition \ref{prop:charKn} are quite different,
it is natural to split each $T \in K_n$ into two halves and investigate
them separately. 

Let $\delta_{n-1}^\L$ be the shape that is the left half of
$\delta_{n-1}$ including the middle column if there is one. In other
words, $\delta_{n-1}^\L$ is the conjugate of $(n-2, n-3, \dots,
\lfloor n/2 \rfloor)$. Note that the shape of the right half of
$\delta_{n-1}$ excluding the middle column is $\delta_{\lfloor n/2
  \rfloor}$. %= \delta_{\lceil \nmo \rceil}$. 
\begin{defn}\label{defn:KnLR}
Let $K_n^\L$ be the set of all the tableaux of shape $\delta_{n-1}^\L$
with integer entries satisfying conditions (a) and (b) of Proposition
\ref{prop:charKn}. For $c =1$ or $2,$ let $K_n^{\L, c}$ be the subset
of $K_n^\L$ consisting all tableaux whose $(1, \lfloor
\nmo\rfloor)$-entry is $c$. (Note that the $(1, \lfloor
\nmo\rfloor)$-entry is the last entry in the first row of any
tableau in $K_n^\L,$ which has to be either 1 or 2 by condition (a)
of Proposition \ref{prop:charKn}.) 

For $c = 1$ or $2,$ let $K_n^{\cR, c}$ be the set of all the tableaux
of shape $\delta_{\lfloor n/2 \rfloor}$ satisfying the following
conditions: 
	\begin{alist}
\itm $T_{1,1} \ge c+1$. 
\itm $T_{a,b} - T_{a,b-1} \ge 1$ for any $2 \le b \le \lfloor n/2
\rfloor-1,\ 1\le a \le \lfloor n/2 \rfloor -b$. 
\itm $T_{a,b} - T_{a-1,b} \ge 2,$ for any $2 \le a \le \lfloor n/2
\rfloor-1,\ 1\le b \le \lfloor n/2 \rfloor-a,$  
\itm $T_{\lfloor n/2 \rfloor-b,b} \le n-1,$ for any $1 \le b \le
\lfloor n/2 \rfloor-1$. %for any $a,b$ satisfying $a+b=\lfloor n/2
                        %\rfloor$. 
	\end{alist}

We consider all the sets above as posets with the coordinate-wise
partial ordering. 
\end{defn}

For any $T \in K_n,$ we define $\Split(T) = (T^\L, T^\cR)$, where $T^L$
of shape $\delta_{n-1}^\L$ is the left half including the middle column of
$T$, and $T^\cR$ of shape $\delta_{\lfloor n/2 \rfloor}$ is the right
half excluding the middle column of $T$. It follows from
Proposition~\ref{prop:charKn} that $(T^\L, T^\cR) \in K_n^{\L, c}
\times K_n^{\cR, c}$ for some $c=1$ or $2$. Thus,  
\[ \Split(K_n) \subseteq \left( K_n^{\L, 1} \times K_n^{\cR, 1}
\right) \bigcupdot \left( K_n^{\L, 2} \times K_n^{\cR, 2}\right).\]  
The equality in the above equation actually holds.
\begin{lem}\label{lem:split}
	We have
\[ \Split(K_n) = \left( K_n^{\L, 1} \times K_n^{\cR, 1} \right)
\bigcupdot \left( K_n^{\L, 2} \times K_n^{\cR, 2}\right).\]  
Therefore,
\[ K_n \cong \left( K_n^{\L, 1} \times K_n^{\cR, 1} \right) \bigcupdot
\left( K_n^{\L, 2} \times K_n^{\cR, 2}\right).\]  
\end{lem}
 
\begin{proof}
We only need to show that for any $(T^\L, T^\cR) \in K_n^{\L, c}
\times K_n^{\cR, c}$ for some $c=1$ or $2,$ if $T$ is the tableau
obtained by gluing $T^\L$ on the left side of $T^\cR$, then  we have
$T \in K_n$. However, by Proposition \ref{prop:charKn}, it is enough
to verify that  
\begin{equation}\label{equ:difflr} T^\cR_{a, 1} - T^\L_{a, \lfloor
    \nmo\rfloor} \ge 1, \qquad \fall a. 
\end{equation}
It follows from conditions (a) and (c) of Definition \ref{defn:KnLR}
that $T^\cR_{a,1} \ge c+1 + 2(a-1)$. Further, because $T^\L$ satisfies
condition (a) of Proposition \ref{prop:charKn} and $T^\L_{1, \lfloor
  \nmo\rfloor} = c,$ we conclude that $T^\L_{a, \lfloor
  \nmo\rfloor} \le c + 2(a-1)$. Therefore equation~\eqref{equ:difflr}
follows. 
\end{proof}

The main goal of this section is to show
\[ K_n^\L \cong M_{\lfloor n/2 \rfloor+1}.\]

\begin{defn} Let $A_n$ be the set of tableaux of shape
  $\delta_{n-1}^\L$ with integer entries satisfying condition (a) of
  Proposition \ref{prop:charKn}.  
	
Write $m=\lfloor \nmo\rfloor$. Let $A'_n$ be the set of tableaux of
rectangular shape 
%$\delta_{n-1}^\L/ \delta_{\lfloor \nmo\rfloor}$  
  \[ (\underbrace{m,m,
    \dots, m }_{ \lfloor n/2 \rfloor}) \] 
  with entries 1 or 2 where the $b$th column has $b$ copies of
  2's and $\lfloor n/2 \rfloor -b$ copies of 1's. 

Let $A''_n$ be the set of tableaux of shape $\delta_{\lfloor n/2
  \rfloor}$ with integer entries in $\{1, 2, \dots, \lfloor n/2
\rfloor\}$ where the entries in each column are strictly increasing.  

Define $\theta_1: A_n \to A'_n$ in the following way. For any $T \in
A_n,$ we do the following three operations on $T$. 
  \begin{enumerate}
    \item For any $1 \le b \le m$ and $1 \le a
      \le m-b,$ we replace the number in the
      $(a,b)$-entry with $T_{a,b}-T_{a-1,b}$. 
    \item Remove all the entries $T_{a,b}$ with $a + b \le
      m$. (Note after this each column has
      $\lfloor n/2 \rfloor$ entries left.) 
    \item Shift all the entries up to make a rectangular shape. 
  \end{enumerate}
We call the resulting rectangular tableau $\theta_1(T)$. 

Define $\theta_2: A'_n \to A''_n$ in the following way. For any $T'
\in A'_n,$ we create $\theta_2(T')$ of shape $\theta_{\lfloor n/2
  \rfloor}$ with entries 
\begin{equation}\label{equ:defntheta2}
\theta_2(T')_{a,b} := \text{ the row index of the $a$th 1 in 
  column $b$ of $T'$}. 
\end{equation}

Define $\theta = \theta_2 \circ \theta_1: A_n \to A''_n$.
\end{defn}

\begin{ex}\label{ex:theta}
Below are examples of the maps $\theta_1$ and $\theta_2$ for $n=6$ and $n=7$.
\[\ytableausetup{centertableaux, boxsize=1.5em}
  \begin{ytableau}
1 & 2 \\
3 & 3 \\
4 & 5 \\
5
\end{ytableau} \ \in A_6  \qquad \stackrel{\theta_1}{\longmapsto} \qquad
  \begin{ytableau}
2 & 2 \\
1 & 1 \\
1 & 2 
\end{ytableau} \ \in A'_6  \qquad \stackrel{\theta_2}{\longmapsto} \qquad
\begin{ytableau}
2 & 2 \\
3 
\end{ytableau} \ \in A''_6 \]

\[\ytableausetup{centertableaux, boxsize=1.5em}
  \begin{ytableau}
1 & 1 & 2\\
2 & 3 & 4\\
4 & 4 & 6\\
5 & 6 \\
6
\end{ytableau} \ \in A_7  \qquad \stackrel{\theta_1}{\longmapsto} \qquad
  \begin{ytableau}
2 & 2 & 2\\
1 & 1 & 2\\
1 & 2 & 2
\end{ytableau} \ \in A'_7  \qquad \stackrel{\theta_2}{\longmapsto} \qquad
\begin{ytableau}
2 & 2 \\
3 
\end{ytableau} \ \in A''_7 \]
\end{ex}

It is clear that $\theta_1$ is a bijection from $A_n$ to $A'_n$ and
$\theta_2$ is a bijection from $A'_n$ to $A''_n$. Hence $\theta$ is a
bijection from $A_n$ to $A''_n$. Below is the main result of this
section. 

\begin{prop}\label{prop:leftiso}
Consider both $A_n$ and $A''_n$ as posets with the coordinate-wise
partial ordering. Then $\theta$ is a poset isomorphism from $A_n$ to
$A''_n$.  

Furthermore (noting that $K_n^\L$ is a subposet of $A_n$) the map
$\theta$ induces a poset isomorphism from $K_n^\L$ to $M_{\lfloor n/2
  \rfloor +1}$. Hence, 
\[ K_n^\L \cong M_{\lfloor n/2 \rfloor+1}.\]
\end{prop}

We will break the proof of Proposition \ref{prop:leftiso} into several
lemmas. For convenience, for any column $C$ with entries in $\{1,2\},$
we define 
\[ \ones(C, i) := \text{the number of 1's in the first $i$ entries of $C$}, \]
and
\[ \rind(C, a) := \text{the row index of the $a$th 1 of $C$}.\]
For example, if $C$ is the first column of the $3 \times 2$ tableau in
$A'_6$ appearing in Example \ref{ex:theta}, we have $\ones(C, 1) = 0,$ $\ones(C,2)=1,$ $\ones(C,3)=2$.  

We have the following obvious lemma on these two statistics.
\begin{lem}\label{lem:onesrind}
	Suppose $C$ and $C'$ are two columns of $\ell$ entries in
        $\{1,2\}$. Then the following two conditions are equivalent. 

	\begin{ilist}
\itm For any $1 \le i \le \ell$, 
%the number of entries of 1's above $i$th row (inclusive) on $C$ is at
%most that on $C'$.
$\ones(C, i) \le \ones(C',i)$.
\itm $\ones(C, \ell) \le \ones(C', \ell)$ and $\rind(C,a) \ge \rind(C', a)$ 
for any $1 \le a \le \ones(C, \ell)$.
%$C$ has at least as many 1s as $C'$ does, and the $a$th 1 on $C$
%appears in a row no lower than the row where the $a$th 1 in $C'$ is
%on, for any $1 \le a \le \#(\text{1's that $C'$ has})$. 
	\end{ilist}
\end{lem}

\begin{proof}
Suppose (i) holds. Clearly we have $\ones(C, \ell) \le \ones(C',
\ell)$. Let $1 \le a \le \ones(C, \ell)$ and $\rind(C, a) =i$. Then 
  $$ a= \ones(C,i) \le \ones(C',i). $$ 
Therefore $\rind(C',a) \le i$, so (ii) follows. 

Suppose there exists $1 \le i \le \ell$ such that $\ones(C, i) > \ones(C', i)$.
Let $a = \ones(C, i)$. Then $\rind(C, a) \le i < \rind(C', a)$.
\end{proof}

Note that with the definition of $\rind$, we can rewrite
\eqref{equ:defntheta2} as
 \[ \theta_2(T')_{a, b} := \rind( \text{column $b$ of $T'$}, a).\]

 \begin{lem}\label{lem:thetaord}
	Suppose $T^{(1)}, T^{(2)} \in A_n$. Then the following
        conditions are equivalent. 
\be
	\itm[\emph{(a)}] $T^{(1)} \le T^{(2)}$. 
	\itm[\emph{(b)}] For any $1 \le j \le \lfloor \nmo \rfloor$ and $1
        \le i \le \lfloor n/2 \rfloor$, 
%the number of entries of 1's above $i$th row (inclusive) in the
%$j$th column of $\sigma_1\left(T^{(1)}\right)$ is at most that in the
%$j$th column of $T^{(2)}$. 
	\[ \ones(\text{column $j$ of $\theta_1(T^{(1)})$}, i) \ge
        \ones(\text{column $j$ of $\theta_2(T^{(2)})$}, i).\] 
	\itm[\emph{(c)}] $\theta(T^{(1)}) \le \theta(T^{(2)})$.
\ee
 \end{lem}

 \begin{proof}
   The equivalence between (b) and (c) follows directly from Lemma
   \ref{lem:onesrind}.  

   Write $m=\lfloor \nmo \rfloor$.
For any $T \in A_n,$ any $1 \le j \le m$ and $1 \le i \le \lfloor n/2
\rfloor$, we have 
   \begin{align*}
 T_{m-j+i,j} - T_{ m-j,j} =& \sum_{a=  m-j+1}^{m-j+i}
 \left(T_{a,j}-T_{a-1,j}\right)\\
   =& \sum_{k=1}^i \theta_1(T)_{k,j} \\
    =& 2i - \ones(\text{column $j$ of $\theta_1(T)$}, i). 
%\# (\text{entries of 1's above $i$th row in the $j$th column of
%$\sigma(T)$}). 
 \end{align*}
 Since $T$ satisfies condition (a)(i) of Proposition \ref{prop:charKn}, we have $T_{ m-j,j} = m-j.$ Thus,
 \begin{equation}\label{equ:theta1}
T_{m-j+i,j} = 2i + m-j - \ones(\text{column $j$ of $\theta_1(T)$}, i).
 \end{equation}
Therefore (a) and (b) are equivalent.
 \end{proof}
 One sees that the first conclusion of Proposition \ref{prop:leftiso}
 follows from Lemma \ref{lem:thetaord}. 

 \begin{lem}\label{lem:thetaweak}
Suppose $T \in A_n$. Then the following conditions are equivalent.
\begin{alist}
\itm $T$ satisfies condition (b) of Proposition \ref{prop:charKn}.
\itm For any $1 \le j \le \lfloor \nmo\rfloor-1$ and $1 \le i \le
\lfloor n/2 \rfloor$, 
%the number of entries of 1's above $i$th row (inclusive) in the
%$j$th column is at most that in the $(j+1)$th column. 
\[ \ones(\text{column $j$ of $\theta_1(T)$}, i) \ge
\ones(\text{column $j+1$ of $\theta(T)$}, i).\] 
\itm The entries are weakly increasing in each row of $\theta(T)$.
\end{alist}
 \end{lem}
 
 \begin{proof}
   The proof is similar to that of Lemma \ref{lem:thetaord}. The
   equivalence between (b) and (c) follows from Lemma
   \ref{lem:onesrind}, and the equivalence between (a) and (b) follows
   from \eqref{equ:theta1}. 
 \end{proof}

 Lemma \ref{lem:thetaweak} implies that $T \in K_n^\L$ if and only if
 $\theta(T) \in M_{\lfloor n/2 \rfloor +1}$. Hence the second part of
 Proposition \ref{prop:leftiso} follows. Below we give a result on the minimal element of $K_n^\L$ that will be used in the next section.

 \begin{lem}\label{lem:flmin}
The last entry in the first row of the unique minimal element of
$K_n^\L$ is 1 if $n$ is even and is 2 if $n$ is odd. 
\end{lem}
\begin{proof}
Note that the minimal element of $M_{\lfloor n/2\rfloor +1}$ is the
tableau whose $(a,b)$-entry is $a$. It is easy to determine the
minimal element in $K_n^\L$ which is in bijection with this minimal
element under the map $\theta,$ and check that it satisfies the
condition described by the lemma. 
\end{proof}

In the rest of this section, we discuss some results on $K_n^{\L, c}$ and $K_n^{\cR,c}.$
 \begin{lem}\label{lem:KnL1even}
   For $m \ge 2,$ 
%the poset isomorphism $\theta: K_{2m}^\L \to M_{m+1}$ induces a poset
%isomorphism from $K_{2m}^{\L, 1}$ to $M_m$. Hence, 
	 \[ K_{2m}^{\L, 1} \cong M_m.\]
 \end{lem}

 \begin{proof} Suppose $T \in K_{2m}^\L$ and $T' = \theta(T)$. 

One checks that $T \in K_{2m}^{\L, 1}$ if and only if the only 1 in
the last column of $\theta_1(T)$ is in the first row, which is
equivalent to the fact that the last entry in the first row of
$T'=\theta(T)$ is $1$.  
 
   However, since $T'$ is an SSYT, the last entry in the first row of
   $T'$ is 1 if and only if the entries in the first row of $T'$ are
   all 1s. There is a natural bijection between tableaux in
   $M_{m+1}$ whose first row is all 1s and tableaux in $M_{m}$:
   taking a tableau in the former set, we remove the first row and
   subtract 1 from each of the remaining entries, and obtain a
   tableau in $M_m$. 
   It is clear that the composition of $\theta$ and this bijection
   gives a poset isomorphism from $K_{2m}^{\L,1}$ to $M_m$. 
%Then our conclusion follows. 
 \end{proof}

 \begin{defn}\label{defn:min}
For $m \ge 1,$ let $\cT_m$ be the tableau of shape $\delta_m$ with entries
\[ (\cT_m)_{a,b} = 2a + b, \qquad \fall 1 \le b \le m-1, \ 1 \le a \le m-b,\]
and $\cT'_m$ the tableau of shape $\delta_m$ with entries
\[ (\cT'_m)_{a,b} = 2a + b-1, \qquad \fall 1 \le b \le m-1, \ 1 \le a \le m-b,\]
\end{defn}

\begin{ex}
Let $m=3$. Then
\[\ytableausetup{centertableaux, boxsize=1.5em}
  \cT_3 = \begin{ytableau}
3 & 4 \\
5  
\end{ytableau}   \qquad \text{ and } \qquad
  \cT'_3 = \begin{ytableau}
2 & 3 \\
4  
\end{ytableau} 
\]
\end{ex}

\begin{lem}\label{lem:rightnonempty} For any $n \ge 2,$ we have
\begin{equation}\label{equ:1contains2}
K_n^{\cR, 1} \supseteq K_n^{\cR, 2}.
\end{equation}
  Further, $\cT_{\lfloor n/2 \rfloor}$ is the unique minimal element
  of $K_n^{\cR, 2}$ and $\cT'_{\lfloor n/2 \rfloor}$ is the unique
  minimal element of $K_n^{\cR, 1}$. 

  Hence, the unique minimal element
  of $K_n^{\cR, 2}$ is greater than the unique
  minimal element of $K_n^{\cR, 1}$.
\end{lem}

\begin{proof}
  \eqref{equ:1contains2} follows directly from the defintion of
  $K_n^{\cR,1}$ and $K_n^{\cR,2}$.  

It is clear that if $\cT_{\lfloor n/2 \rfloor} \in K_n^{\cR, 2}$, it
has to be the unique minimal element of $K_n^{\cR, 2}$. Hence we only
need to show that  
  $\cT_{\lfloor n/2 \rfloor} \in K_n^{\cR, 2}$, which can be proved by
verifying conditions (a)--(d) of Definition \ref{defn:KnLR}.  
  We can similarly prove the statement on $\cT'_{\lfloor n/2 \rfloor}$.
\end{proof}

 \begin{lem}\label{lem:rightiso}
   Let $m \ge 1$.
  \begin{alist}
    \itm $K_{2m}^{\cR, 1} \cong M_{m+1}$. 
    \itm $K_{2m+1}^{\cR, 2} \cong M_{m+1}$.
  \end{alist}
\end{lem}

\begin{proof}
	We first prove (a). For any $T \in M_{m+1},$ we define
        $\phi(T)$ to be the tableau of shape $\delta_{m}$ with entries 
	\[ \phi(T)_{a,b} = T_{a,b} + a+ b-1.\]
Comparing Lemma \ref{lem:Mnentry} and Definition \ref{defn:KnLR}, one
sees that $\phi(T) \in K_{2m+1}^{\cR, 2}$. Hence $\phi: M_{m+1} \to
K_{2m}^{\cR, 1}$. It is easy to define the inverse map of $\phi$ and
verify that $\phi$ is an poset isomorphism. 

b) can be proved similarly by defining a map $\phi': M_{m+1} \to
K_{2m+1}^{\cR, 2}$ where 
	\[ \phi'(T)_{a,b} = T_{a,b} + a+ b.\]
\end{proof}

%We have an immediate consequence of Proposition \ref{prop:leftiso}
%and Lemma \ref{lem:rightnonempty}. 
We finish this section by concluding the nonemptyness of $K_n$, which leads to Elkie's formula for $\sigma(n)$ as we've discussed in Remark \ref{rem:elkies}.
\begin{cor}
For any $n \ge 2,$ the poset $K_n$ is nonempty. 
\end{cor}

\begin{proof}
By Proposition \ref{prop:leftiso} and Lemma \ref{lem:rightnonempty},
the sets $K_n^\L$, $K_n^{\cR,1}$ and $K_n^{\cR,2}$ are all
nonempty. Therefore the conclusion follows from Lemma
\ref{lem:split}. 
\end{proof}

%Therefore, as we've discussed in Remark \ref{rem:elkies}, we recover
%Elkies' formula for $\sigma(n)$. 
\begin{cor}[Elkies] For $m \ge 1,$
  \[ \sigma(n) = \sum_{b=1}^{n-2} \min(b, n-1-b) =\begin{cases}
		m(m-1), & \text{if $n=2m$}; \\
		m^2, & \text{if $n=2m+1$}.
	\end{cases}\]
\end{cor}

\section{Join-irreducibles of $K_n$} \label{sec5}
  
In the last section we confirmed Elkies' formula for $\sigma(n)$. As a
consequence, the value $g(n)$ is the cardinality of $K_n$: $g(n) = \#
K_n$.  In this section, we will determine rank-generating function
of $K_n$, which leads to a formula for $\# K_n$ by discussing the
structure of the poset of join-irreducibles of $K_n$.

Let $U_n$ be the poset of the join-irreducibles of $K_n$.  By a result
of Dilworth (see Exercise 3.72(a) of \cite{ec1}), $K_n$ is a
distributive lattice. Hence by the fundamental theorem for finite
distributive lattices, we have $K_n = J(U_n)$.

Suppose $T \in K_{n}$ and $\Split(T) = (T^\L, T^\cR)$. Let $c$ be the
last entry in the first row of $T^\L$. It is clear that $T$ is a
join-irreducible if and only if one of the following two cases
happens: 
\begin{enumerate}
	\item $T^\L$ is a join-irreducible of $K_{n}^\L,$ and $T^\cR$
          is the unique minimal element in $K_{n}^{\cR,c}$.
	\item $T^\L$ is the unique minimal element in $K_{n}^\L$, and
          $T^\cR$ is a join-irreducible of $K_{n}^{\cR, c}$. 
\end{enumerate}

We call $T \in K_n$ a \emph{left-join-irreducible} if it fits into
situation (1), and a \emph{right-join-irreducible} if it fits into
situation (2). Let $U_n^\ell$ ($U_n^r$, respectively) be the subposet
of $U_n$ that consists of all the left-join-irreducibles
(right-join-irreducibles, respectively). Further, for $i=1,2$ we let
$U_n^{\ell,c}$ be the set of those $T$ that fit into situation (1)
with $c=i$.  (Thus $U_n^\ell$ is the disjoint union of $U_n^{\ell, 1}$
and $U_n^{\ell,2}$.)

Note that for a given $n,$ since the unique minimal element in
$K_n^\L$ is fixed the number $c$ in situation (2) is fixed. Using
Lemma \ref{lem:flmin}, we can give a more explicit description of
$U_n^r$ depending on the parity of $n$. 

\begin{lem}\label{lem:Unr}
 Suppose $T \in K_n$ and $\Split(T) = (T^\L, T^\cR)$. %Assume $T^\L$
                                %is the unique minimal element in
                                %$K_n^\L$.  
 \be
   \itm[\emph{(a)}] For even $n$, $T \in U_n^r$ if and only if $T^\L$
   is the 
   unique minimal element in $K_n^\L$ and $T^\cR$ is a
   join-irreducible of $K_n^{\cR, 1}$. 
   \itm[\emph{(b)}] For odd $n$, $T \in U_n^r$ if and only if $T^\L$ is the unique
   minimal element in $K_n^\L$ and $T^\cR$ is a join-irreducible of
   $K_n^{\cR, 2}$. 
 \ee
\end{lem}

Recall that $Q_n$ is the poset of join-irreducibles of $M_n.$
\begin{lem}\label{lem:Uniso} For $n \ge 2,$ the following are true.
\be
\itm[\emph{(a)}] $\displaystyle U_n^\ell \cong Q_{\lfloor n/2 \rfloor +1}$.
\itm[\emph{(b)}] $U_n^{\ell,1} \cong $ the poset of join-irreducibles of
$K_n^{\L, 1}$. 
\itm[\emph{(c)}] $\displaystyle U_n^r \cong Q_{\lfloor n/2 \rfloor +1}$.
\ee
\end{lem}

\begin{proof}
  Suppose for $i=1,2$ we have  $T^{(i)} \in U_n^\ell$ and
  $\Split(T^{(i)}) = 
  (T^{(i),\L}, T^{(i),\cR})$. By Lemma \ref{lem:rightnonempty}, we have $T^{(1)} \le T^{(2)}$ if
  and only if $T^{(1), \L} \le T^{(2),\L}$. Therefore $U_n^\ell$ is
  isomorphic to the poset of join-irreducibles of $K_n^\L$. 
%which is isomorphic to $Q_{\lfloor n/2\rfloor +1}$ following
%Proposition \ref{prop:leftiso}. 
  Hence (a) follows from Proposition \ref{prop:leftiso}. 

For any $T \in K_n^{\L, 1},$ because it does not cover any element in
$K_n^{\L,2}$ in the poset $K_n^{\L},$ the elements covered by $T$ in
$K_n^{\L}$ are exactly the same as the elements covered by $T$ in
$K_n^{\L,1}$. Hence (b) follows from the fact that $U_n^\ell$ is
isomorphic to the poset of join-irreducibles of $K_n^\L$.

  By Lemma \ref{lem:Unr}, $U_n^r$ is isomorphic to the poset of
  join-irreducibles of $K_n^{\cR,1}$ if $n$ is even and isomorphic to
  the poset of join-irreducibles of $K_n^{\cR,2}$ if $n$ is odd. It
  follows from Lemma \ref{lem:rightiso} that $U_n^r \cong Q_{\lfloor
    n/2 \rfloor +1}$.
\end{proof}

\begin{lem}\label{lem:charcomp}
 Suppose $T^{\ell} \in U_n^\ell$ and $T^{r} \in U_n^r$ are
 left-join-irreducible and right-join-irredubcile of $K_n,$
 respectively. For $s= \ell, r,$ let $\Split(T^{s}) = (T^{s, \L},
 T^{s, \cR})$ and $c^s$ the last entry in the first row of $T^{s,\L}$.
 Then the following are equivalent.

	\begin{ilist}
	\itm $T^\ell$ and $T^r$ are comparable.
	\itm $T^r < T^\ell$. % (as elements in $U_n$).
	\itm $c^r = 1,$ $c^\ell =2,$ and $T^{r, \cR} \le \cT_{\lfloor
          n/2 \rfloor}$. % (as elements in $K_n^{\cR, 1}$).  
	(Recall that $\cT_{\lfloor n/2 \rfloor}$ defined in Definition
        \ref{defn:min} is the unique minimal element of $K_n^{\cR,
          2}$.) 
	\end{ilist}
\end{lem}

\begin{proof}
  Since $T^{r, \L}$ is the minimal element of $K_n^\L,$ we clearly
  have $T^{r, \L } < T^{\ell, \L}$. Therefore (i) and (ii) are
  equivalent. Furthermore, we have condition (ii) is equivalent to
  $T^{r, \cR} \le T^{\ell, \cR}$.

  Note that if $c^\ell=2,$ $T^{\ell, \cR} = \cT_{\lfloor n/2
    \rfloor}$. Hence (iii) implies (ii). Now we assume (ii) which
  implies $T^{r, \cR} \le T^{\ell, \cR}$. 
  %We check that either $c^r = c^\ell$ or $c^\ell=1$ implies that $T^{r, \cR} > \cT_{\lfloor n/2 \rfloor}$. 
    If $c^r \ge c^\ell,$ then by Lemma \ref{lem:rightnonempty},
    \begin{align*}
      T^{r, \cR} >& \text{the unique minimal element in $K_n^{\cR, c^r}$} \\
      \ge& \text{the unique minimal element in $K_n^{\cR, c^\ell}$} = T^{\ell, \L}.
    \end{align*}
    Thus we must have $c^r = 1$ and $c^\ell =2$. Then
  $T^{\ell, \cR} = \cT_{\lfloor n/2 \rfloor}$ and (iii) follows.
\end{proof}

Because of condition (iii) in Lemma \ref{lem:charcomp}, it is natural
for us to divide $U_n^r$ into two sets as well.
\begin{defn}
  Let $U_n^{r,2}$ be the subposet of $U_n^r$ that consists of all 
  tableaux $T$ such that $T^\cR \le \cT_{\lfloor n/2 \rfloor},$ where
  $\Split(T) = (T^\L, T^\cR)$.

  Let $U_n^{r,1}$ be the set $U_n^{r} \setminus U_n^{r,2}$ with the
  coordinate-wise partial ordering. 
\end{defn}

Note that in Lemma \ref{lem:charcomp}, $T^r$ is a
right-join-irreducible and thus $c^r$ is the last entry in the first
row of the minimal element of $K_n^\L$. Thus by Lemma
\ref{lem:flmin}, $c^r$ is always 1 for even $n$ and is always 2
for odd $n$. Applying Lemma \ref{lem:charcomp} to odd cases and even
cases separately, we have the following results.

\begin{cor} \label{cor:charcomp}
  Suppose $T^{\ell} \in U_n^\ell$ and $T^{r} \in U_n^r$ are
  left-join-irreducibles and right-join-irredubciles of $K_n,$
  respectively.
\be
\itm[\emph{(a)}] If $n$ is odd, then $T^\ell$ and $T^r$ are
incomparable.  
\itm[\emph{(b)}]
Suppose $n$ is even. Let $\Split(T^{r}) = (T^{r, \L}, T^{r,
  \cR})$. Then the following are equivalent.

	\begin{ilist}
	\itm $T^\ell$ and $T^r$ are comparable.
	\itm $T^r < T^\ell$. % (as elements in $U_n$).
	\itm $T^\ell \in U_n^{\ell,2}$ and $T^r \in U_n^{r,
          2}$. %$T^{r, \cR} \le \cT_{\lfloor n/2 \rfloor}$. % (as
               %elements in $K_n^{\cR, 1}$).  
	\end{ilist} 
\ee
\end{cor}

We now have enough information to determine the rank-generating
function of $K_n$ for odd $n$.  
\begin{thm} \label{thm:odd}
Suppose $n=2m+1$ for some $m \ge 1$. Then 

\begin{equation}\label{equ:Knodd}
K_n \cong M_{m+1} \times M_{m+1} \cong J(Q_{m+1} + Q_{m+1}).
\end{equation}
Therefore the rank-generating function of $K_n$ is given by
\begin{equation}\label{equ:genodd} F(K_n,q) =
  \left((1+q)^{m-1}(1+q^2)^{m-2}\cdots (1+q^{m-1})\right)^2, 
\end{equation}
where $F(K_3,q)=1$.
\end{thm}

\begin{proof}
Corollary \ref{cor:charcomp} implies that for odd $n,$ we have as posets,
\[ U_n = U_n^\ell + U_n^r.\]
Then \eqref{equ:Knodd} follows from Lemma \ref{lem:Uniso}. Applying
Theorem \ref{thm1}, we obtain \eqref{equ:genodd}. 
\end{proof}
\begin{rem}
  Theorem \ref{thm:odd} can be proved more directly without discussing
  the poset $U_n$ of join-irreducibles. One can argue that for odd
  $n,$ the set $K_n^{\L,1}$ is empty and thus $K_n \cong K_n^{\L,2}
  \times K_n^{\cR,2} = K_n^{\L} \times K_n^{\cR,2}$. Then \eqref{equ:Knodd} follows from Proposition
  \ref{prop:leftiso} and Lemma \ref{lem:rightiso}. 
\end{rem}

Substituting $q=1$ in \eqref{equ:genodd} gives us the cardinality of
$K_n,$ which is the value of $g(n)$. 
\begin{cor}Suppose $n=2m+1$ for some $m \ge 1$. Then
\[ g(n) = \# K_n = 2^{m(m-1)}.\]
\end{cor}

%\begin{lem}\label{lem:oddcenter}
%	Suppose $n=2m+1$ for some $m \ge 1$ and $T \in K_n$. Then the entries on the $m$th column of $T$ are $2, 4, 6, \dots, 2m$.
% \end{lem}

%\begin{proof}
%	This follows immediately from Proposition \ref{prop:charKn}(a)(ii) with $b = m$.
%\end{proof}

%\begin{proof}[Proof of Theorem \ref{thm:odd}]
%	It follows from Lemma \ref{lem:oddcenter} that $K_n^{\L, 1} = \emptyset$. Thus by Lemma \ref{lem:split}, Proposition \ref{prop:leftiso} and Lemma \ref{lem:rightiso}(b)
%	\[ K_{2m+1} \cong K_{2m+1}^{\L,2} \times K_{2m+1}^{\cR, 2} \cong M_{m+1} \times M_{m+1}.\]
%	Then \eqref{equ:genodd} follows from Theorem \ref{thm1}.
%\end{proof}

We focus on the case where $n$ is even for the rest of this section. 

\begin{lem}\label{lem:U2mr1}
For $m \ge 2$ we have $U_{2m}^{r,1} \cong Q_m$.
\end{lem}

\begin{proof}
First, it's clear that 
\begin{equation}\label{equ:U2mr1} U_{2m}^{r,1} \cong \{ T \not\le
  \cT_m \ | \ \text{$T$ is a join-irreducible of $K_{2m}^{\cR,1}$}\}, 
\end{equation}
where $\cT_m$ as defined in Definition~\ref{defn:min} is the minimal
element of $K_{2m}^{\cR,2}$.
Recall that in Proposition \ref{prop:Phin} we define a poset $\Phi_n$
and a poset isomorphism $\psi\colon\Phi_n\to Q_n,$ and in the
proof of Lemma \ref{lem:rightiso} we define a poset isomorphism
$\phi\colon M_{m+1}\to K_{2m}^{\cR,1}$. Letting $n = m+1$ and taking
the composition of $\psi$ and $\phi,$ we obtain an isomorphism from
$\Phi_{m+1}$ to the poset of join-irreducibles of $K_{2m}^{\cR,1}$.
Further, it is easy to see that 
 \beas \phi(\psi((a, b, k)) & = & \phi ( \Add(\cT^0_{m}; a, b,
 k))\\ & = & \Add(\phi(\cT_{m}^0); a, b, k)\\ 
  & = & \Add(\cT'_m; a, b, k), \eeas
where $\cT'_m$ as defined in Definition~\ref{defn:min} is the minimal
element of $K_{2m}^{\cR,1}$. Comparing the definitions of $\cT'_m$ and
$\cT_m,$ one sees that $\cT_m = \Add(\cT'_m; 1, 1, 1)$.
Hence, 
\[ \Add(\cT'_m; a, b, k) \not\le \cT_m \iff k \ge 2.\]
Therefore the right-hand side of \eqref{equ:U2mr1} is isomorphic to
\begin{align*}
\{ (a, b, k) \in \Phi_{m+1} \ | \ k \ge 2 \} =& \{ (a,b,k) \in \P^3
\ | \ 2 \le k \le b \le m-a\} \\ 
\cong& \{ (a,b',k') \in \P^3 \ | \ 1 \le k' \le b' \le m-1-a\}\\
 &  \quad (k'=k-1, b'=b-1) \\ =& \Phi_m,
\end{align*}
which is isomorphic to $Q_m$ by Proposition \ref{prop:Phin}.
\end{proof}

At this point, we have a good understanding of the structure of $U_n$
for even $n$. We summarize this in the following proposition. 
\begin{prop}\label{prop:Uneven} Suppose $n=2m$ for $m \ge 2$.
  The poset $U_n$ of the join-irreducibles of $K_n$ can be divided
  into two disjoint sets $U_n^\ell$ and $U_n^r$, each of which are
  divided into two disjoint sets $U_n^\ell = U_n^{\ell,1} \bigcupdot
  U_n^{\ell,2}$ and $U_n^r = U_n^{r,1} \bigcupdot U_n^{r,2}$ such that
  they satsify the following conditions:  
\begin{alist}
  \itm $U_n^\ell \cong U_n^r \cong Q_{m+1}$.
  \itm $U_n^{\ell, 1} \cong U_n^{r,1} \cong Q_m$.
  \itm No element in $U_n^{\ell,1}$ is comparable to any element in
  $U_n^r$. 
  \itm No element in $U_n^{r,1}$ is comparable to any element in
  $U_n^\ell$. 
  \itm Each element of $U_n^{r,2}$ is smaller than any element in
  $U_n^{\ell,2}$. 
\end{alist}
\end{prop}

\begin{proof}
 (a) follows from Lemma \ref{lem:Uniso}(a,c), and (c)--(e) follow from
  Corollary \ref{cor:charcomp}(b). Finally, (b) follows from Lemma
  \ref{lem:Uniso}(b), Lemma \ref{lem:KnL1even}, and Lemma
  \ref{lem:U2mr1}.  
\end{proof}

\begin{thm} \label{thm:even}
Suppose $n=2m$ for some $m \ge 2$. Then 
the rank-generating function of $K_n$ is given by
 \beas\label{equ:geneven} F(K_n,q) & = &
  \left((1+q)^{m-2}(1+q^2)^{m-1}\cdots (1+q^{m-2})\right)^2\\ & & \ \ 
  \cdot \left( (1+q) (1+q^2) \cdots (1+q^{m-1}) \times
  \left(1+q^{\binom{m}{2}}\right) - q^{\binom{m}{2}} \right), \eeas
where $F(K_4,q)= (1+q)^2 - q = 1+q+q^2$.
\end{thm}

\begin{proof}
  The part of $F(K_n, q) = F(J(U_n),q)$ which corresponds to order ideals
  that do not contain any element of $U_n^{\ell,2}$ is 
  \bea\label{equ:part1}
    F(J(U_n^r + U_n^{\ell,1}), q) & = & F(J(U_n^r),q)
    F(J(U_n^{\ell,1}),q) \nonumber \\ 
    & = & F(M_{m+1},q) F(M_{m},q), \eea
  and the part corresponding to order ideals that contain at least one
  element from $U_n^{\ell,2}$ (and thus contain all the elements in
  $U_n^{r,2}$) is 
 \begin{align}   & \left( F(J(U_n^\ell),q) - F(J(U_n^{\ell,1}),
   q)\right) \times q^{\# U_n^{r,2}} \times F(J(U_n^{r,1}), q)
   \nonumber \\ 
   =&  \left( F(M_{m+1},q) - F(M_m,q) \right) \times q^{\# Q_{m+1} -
     \# Q_m} \times F(M_m, q).\label{equ:part2} 
 \end{align}
 We obtain the formula for $F(K_n, q)$ by adding \eqref{equ:part1} and
 \eqref{equ:part2}, and then substituting from formulas
 \eqref{equ:genMn} and  \eqref{equ:diffQn}. 
\end{proof}

\begin{cor}Suppose $n=2m$ for some $m \ge 1$. Then
\[ g(n) = \# K_n = 2^{(m-1)(m-2)}(2^m-1).\] 
\end{cor}

\begin{proof}
If $m=1,$ one checks directly that $K_n$ contains one element. Hence,
$g(2) = \# K_2 = 1 = 2^0 (2^1-1)$. 

For $m \ge 2$ the conclusion follows from substituting $q=1$ in the formula for $F(K_n, q)$ given in Theorem \ref{thm:even}.
%\eqref{equ:geneven}. 
\end{proof}


\begin{thebibliography}{99}
\small \setlength{\itemsep}{-.8mm}

\bibitem{blog} G. Antonick, \emph{The Improbable Life of Paul
  Erd\H{o}s}, \emph{New York Times} Wordplay blog, March 25, 2013,
  $$
  \mbox{\texttt{http://wordplay.blogs.nytimes.com/2013/03/25/erdos}.} $$ 
%
 \bibitem{propp} J. Propp (moderator), Domino forum,
 $$
   \mbox{\texttt{http://faculty.uml.edu/jpropp/about-domino.txt}.} $$  
%
\bibitem{ec1} R. P. Stanley, \emph{Enumerative Combinatorics},
  vol.~1, second ed., Cambridge Studies in Advanced Mathematics, vol. 49,
  Cambridge University Press, Cambridge, 2012.
%
\bibitem{ec2} R. P. Stanley, \emph{Enumerative Combinatorics},
  vol.~2, Cambridge Studies in Advanced Mathematics, vol. 62,
  Cambridge University Press, Cambridge, 1999.

\end{thebibliography}
\end{document}